\documentclass[12pt]{amsart}

\usepackage{wasysym}
\usepackage[final]{optional}

\usepackage{tikz}

\usepackage{cancel}
\usepackage{latexsym, amssymb, amsmath}
\usepackage{soul,esint}
\usepackage{amsfonts, graphicx}
\usepackage{graphicx,color}
\usepackage{mathtools}
\usepackage{hyperref}
\usepackage{verbatim}
\makeatletter
\renewcommand*{\eqref}[1]{%
  \hyperref[{#1}]{\textup{\tagform@{\ref*{#1}}}}%
}
\makeatother

\newcommand{\be}{\begin{equation}}
\newcommand{\ee}{\end{equation}}
\newcommand{\beq}{\begin{eqnarray}}
\newcommand{\eeq}{\end{eqnarray}}

\newtheorem{thm}{Theorem}[section]

\newtheorem{lma}{Lemma}[section]

\theoremstyle{remark}
\newtheorem{rem}{Remark}[section]
\newtheorem{question}{Question}[section]
\numberwithin{equation}{section}

\def\be{\begin{equation}}
\def\ee{\end{equation}}
\def\bee{\begin{equation*}}
\def\eee{\end{equation*}}

\def\e{\varepsilon}

\def\ga{{\gamma}}

\def\R{\mathbb{R}}

\begin{document}
\title[]
{Metric limits of manifolds with \\ positive scalar curvature}

\author{Man-Chun Lee}
\address[Man-Chun Lee]{Department of Mathematics, The Chinese University of Hong Kong, Shatin, N.T., Hong Kong
}
\email{mclee@math.cuhk.edu.hk}

\author{Peter M. Topping}
\address[Peter M. Topping]{Mathematics Institute, Zeeman Building, University of Warwick, Coventry CV4 7AL}
\urladdr{https://warwick.ac.uk/fac/sci/maths/people/staff/peter\_topping/}


\date{14 November 2024}

\begin{abstract}
We show that any Riemannian metric conformal to the round metric on $\mathbb{S}^{n}$, for $n\geq 4$, arises as a limit of a sequence of Riemannian metrics of positive scalar curvature on $\mathbb{S}^{n}$ in the sense of uniform convergence of  Riemannian distance.
In particular, non-negativity of scalar curvature is not preserved under such limits.
\end{abstract}


\maketitle{}

\markboth{}{Metric limits of manifolds with positive scalar curvature}

\section{Introduction}

Given a sequence of Riemannian manifolds satisfying some notion of positive curvature, 
that Gromov-Hausdorff converge to some other Riemannian manifold, one can ask whether the limit manifold inherits the nonnegativity of curvature. 
When one is considering positive sectional curvature then the theory of Alexandrov spaces can be invoked to give positive results of this form \cite{BBI}.
In the case of positive Ricci curvature then optimal transportation techniques
can be applied, see e.g. \cite{villani}.
For manifolds satisfying the PIC1 condition, the Ricci flow will give a positive answer to this question \cite{LT1}. 

Positive scalar curvature is too weak a condition to give such results. 
An arbitrary $n$-dimensional Riemannian manifold can be approximated by a graph
in the Gromov-Hausdorff sense, and the graph can be `fattened' to a nearby 
$n$-dimensional Riemannian manifold of positive scalar curvature by replacing the vertices by small spheres and replacing the edges by even smaller tubes (see e.g. 
\cite{gromov_survey}).
In such an example, the topology of the approximating manifolds is becoming infinitely complex in the limit.

In this note we consider what happens when we strengthen the notion of convergence so that the topology of the underlying manifolds is fixed. We then ask that the Riemannian distances converge uniformly. This type of convergence arises naturally in Ricci flow theory.
We prove that metrics of positive scalar curvature on a sphere of dimension at least four can be made to converge in this sense
to a completely arbitrary metric on the same sphere that is conformal to the round metric.

In the following we write $g_{0}$ for the standard round metric
on $\mathbb{S}^{n}$. We write $\mathcal{R}(g)$ for the scalar curvature of a metric $g$. 
\begin{thm}
\label{main}
Suppose 
$n\geq 4$, and $f\in C^0({\mathbb{S}^{n}})$. Then there exists a sequence of Riemannian metrics $g_i$ on ${\mathbb{S}^{n}}$ such that $\mathcal{R}(g_i)>0$ and 
such that $d_{g_i}$ converges uniformly on ${\mathbb{S}^{n}}\times {\mathbb{S}^{n}}$ to the Riemannian distance $d_f$ of the metric $e^{2f}g_{0}$. 
In particular, $({\mathbb{S}^{n}},d_{g_i})$ converges to $({\mathbb{S}^{n}},d_f)$ in the Gromov-Hausdorff sense as $i\to \infty$ with the same underlying manifold ${\mathbb{S}^{n}}$ for each $i$.
Moreover, there exists $C<\infty$ such that 
$$\frac{g_0}{C}\leq g_i\leq Cg_0$$
on $\mathbb{S}^n$ for all $i$.
\end{thm}

Clearly, for some choices of conformal factor $f$ the scalar curvature of $e^{2f}g_{0}$ will fail to be everywhere nonnegative.  


It may be worth stressing that in the theorem we are considering $C^0$ convergence of the Riemannian distance and not the  stronger notion of $C^0$ convergence of the Riemannian metric. In the latter case Gromov and Bamler \cite{gromov, bamler} showed that  nonnegativity of scalar curvature \emph{would} be inherited by the limit; see also \cite{Burkhardt}. 

In order to construct the approximating metrics $g_i$, we will take a tightly packed collection of great circles in the sphere and carefully shrink the metrics along the great circles to be more like they would be on the desired limit metric 
$e^{2f}g_{0}$, without breaking the positive scalar curvature. 
Given two points in an approximating $({\mathbb{S}^{n}},g_i)$, 
an almost-minimising path can then be constructed that makes many short trips along different great circles, broadly following the path of a minimising geodesic
within $({\mathbb{S}^{n}},e^{2f}g_{0})$. 
The shrinking process along great circles will be described in Section \ref{shrink_sect}.

The same techniques will yield the following analogue on the torus.
This time $g_0$ is any flat metric on $T^{n}$.
\begin{thm}
\label{main_torus}
Suppose 
$n\geq 4$, and $f\in C^0({T^{n}})$. Then there exists a sequence of Riemannian metrics $g_i$ on ${T^{n}}$ such that $\mathcal{R}(g_i)>-\frac{1}{i}$ and 
such that $d_{g_i}$ converges uniformly on ${T^{n}}\times {T^{n}}$ to the Riemannian distance $d_f$ of the metric $e^{2f}g_{0}$. 
The metrics $g_i$ are uniformly equivalent to $g_0$ in the sense that for some $C<\infty$
we have $\frac{g_0}{C}\leq g_i\leq Cg_0$ for all $i$.
\end{thm}
In particular, the limit need  not be the flat metric as it would have to be for even stronger notions of convergence of $g_i$; see \cite{Burkhardt, bamler, gromov}. 
We will summarise the changes required to prove Theorem \ref{main_torus}
in Section \ref{torus_sect}. 
Given the form of the desired metric $e^{2f}g_0$, it is natural to ask whether the approximating metrics $g_i$ 
can also be taken to be conformal to $g_0$.
It follows from  work of Chu and the first named author \cite{ChuLee} that if a sequence of metrics $g_i$
on $T^n$ satisfy $\mathcal{R}(g_i)>-\frac{1}{i}$ 
and are all uniformly equivalent to $g_0$, as 
in Theorem~\ref{main_torus}, but are additionally all 
conformally equivalent, then $g_i$ will sub-converge to a flat torus in the measured Gromov-Hausdorff sense,
in stark contrast to Theorem~\ref{main_torus}.

\begin{rem}
Because the metrics $g_i$ in Theorems \ref{main} and \ref{main_torus} are uniformly bi-Lipschitz,
it follows from \cite[Theorem A.1]{HuangLeeSormani} that $g_i$ also converges to $e^{2f}g_0$ in the intrinsic flat sense. 
\end{rem}

Theorems \ref{main} and \ref{main_torus} require the dimension of the manifold to be at least four in order to shrink the metric in a given direction while preserving the scalar curvature lower bound based on a 
construction of the first author, Naber and Neumayer \cite{LNN}. This begs the question of what happens in three dimensions.
In contrast to Theorem \ref{main} we have:
\begin{question}
\label{Q1}
Suppose $g_i$ is a sequence of smooth metrics with $\mathcal{R}(g_i)\geq 0$ on a three-dimensional manifold $M^3$, such that $d_{g_i}$ converges uniformly on $M\times M$ to $d_{g_0}$, where $g_0$ is another smooth metric on $M$. 
Is it true that $\mathcal{R}(g_0)\geq 0$?
\end{question}

As mentioned above, if the notion of metric convergence is weakened to allow the topology to change then non-negative scalar curvature is not preserved even in three dimensions. Further results in the case that the topology is allowed to vary can be found in the work of Basilio-Sormani \cite{BasilioSormani}, Basilio-Dodziuk-Sormani \cite{BasilioDodziukSormani} and Basilio-Kazaras-Sormani \cite{BasilioKazarasSormani}. 
Very recently, the collapsing example in \cite{LNN} was generalized by  Kazaras-Xu \cite{KazarasXu} to three dimensions. Based on this, we expect Question \ref{Q1} to be false in general.
%
%
For a large number of other questions and problems concerning the issues addressed by Question \ref{Q1}, see \cite{gromov_survey} and \cite{Sormani}.

\bigskip
\noindent
\emph{Acknowledgements:} The authors would like to thank Brian Allen, Christina Sormani and Misha Gromov for useful comments. PT was supported by EPSRC grant EP/T019824/1.

\section{The building block}
\label{shrink_sect}
The essential building block used in the proof of Theorem \ref{main} is the following:

\begin{lma}
\label{building_block_lem}
Suppose $n\geq 4$, $\mathcal{C}$ is a great circle in $({\mathbb{S}^{n}},g_{0})$, 
$R\in (0,\frac{1}{100})$, and write $\mathcal{C}_R$ for the $R$-tubular neighbourhood of $\mathcal{C}$.
Suppose $f\in C^\infty({\mathbb{S}^{n}})$ with $f\leq -1$, and 
define $\bar{f}=\max (-f)$, so that $f\in [-\bar{f},-1]$.
Then for each $\e\in (0,1)$ we can find a new smooth Riemannian metric $g$ on ${\mathbb{S}^{n}}$ with the properties that 
\begin{enumerate}
\item
$\mathcal{R}(g)>0$
\item
$g=g_{0}$ outside $\mathcal{C}_R$
\item
$g\leq (1+\e)g_{0}$ throughout ${\mathbb{S}^{n}}$
\item
$e^{2f}g_{0}\leq (1+\e)g$, and in particular 
$e^{-2\bar{f}}g_{0}\leq (1+\e)g$
\item
The metrics on $\mathcal{C}$ induced by restricting 
$e^{2f}g_{0}$ and $g$ are equal. 
That is, both $e^{2f}g_{0}$ and $g$ agree on the length of   vectors tangent to $\mathcal{C}$.
\end{enumerate}
\end{lma}
The lemma tells us that we can pull the great circle $\mathcal{C}$ tight so that distances along it are reduced to what they would be with respect to the shrunk metric $e^{2f}g_{0}$, 
without breaking the positive scalar curvature, and without changing the metric far from $\mathcal{C}$. 
This builds on a construction of the first author, Naber and Neumayer in \cite{LNN}.

\bigskip

\begin{proof}[Proof of Lemma \ref{building_block_lem}]
We begin by showing how the standard round metric on ${\mathbb{S}^{n}}$ can be scarred along the great circle $\mathcal{C}$ by pulling out a mountain ridge along 
$\mathcal{C}$
that is not very high or steep, but has very positive curvature along the top of the ridge. 
Thus locally to $\mathcal{C}$ the metric will look like the product of an interval with a cone over an $(n-2)$-dimensional sphere of radius a little below $1$, slightly smoothed.
The high curvature near the cone point, i.e., along the ridge, gives us some leeway to modify the metric.
In particular, it will allow us to substantially shrink the metric along the ridge without destroying the positive scalar curvature.

We will treat the sphere ${\mathbb{S}^{n}}$ as a doubly warped product manifold: Consider 
$$(r,x,y)\mapsto  \left(x\cdot \sin r, y\cdot \cos r\right) \in 
\mathbb{S}^{n}\subset \mathbb{R}^{n+1},$$
where $r\in (0,\frac\pi{2})$,  $x\in \mathbb{S}^{n-2}\subset \mathbb{R}^{n-1}$ and $y \in \mathbb{S}^1\subset\mathbb{R}^2$.  In these coordinates, the spherical metric $g_{0}$ can be represented by 
\begin{equation}
\label{sph_metric}
g_{0}=dr^2+ \sin^2 r \cdot h_{\mathbb{S}^{n-2}}+\cos^2 r \cdot ds^2
\end{equation}
where we write $h_{\mathbb{S}^{n-2}}$ to denote the spherical metric on $\mathbb{S}^{n-2}$.  

We would like now to modify the round metric in a tubular neighbourhood of the great circle $\{r=0\}$,
which we may assume to be $\mathcal{C}$.
. We will adjust it within the region where $r\leq R$.
It suffices to prove the lemma for any smaller $R>0$ than given, so we take the opportunity to appeal to the uniform continuity of $f$ and reduce $R$ so that  
\begin{equation}
\label{R_unif_cty}
\textstyle d_{g_{0}}(x,y)\leq R\qquad\text{implies}\qquad |f(x)-f(y)|\leq \frac12 \log(1+\e).
\end{equation}

The modified metric will take the form
\begin{equation}
\label{ansatz}
g=\frac{dr^2}{\alpha(r)}+ \sin^2 r \cdot h_{\mathbb{S}^{n-2}}+e^{2\beta (s,r)} \cdot  ds^2.
\end{equation}
The (smooth) function $\alpha(r)$ is illustrated in Figure \ref{fig:alpha}.
The functions $\alpha$ and $\beta$ are specified and constrained as follows. 
Define 
\begin{equation}
\label{hat_alpha_choice}
\textstyle\hat\alpha=\max\{\frac{3}{4}, 1-\frac{\e}{2}, 1-\frac{R^2}5\} <1.
\end{equation}
For $\epsilon_0\in (0,R/2)$ to be chosen later, we divide the interval $[0,R]$ into three zones and insist on the following properties:
\begin{enumerate}
\item
Zone 1: For $r\in [0,\epsilon_0]$ we ask that $\alpha(r)\equiv 1$ for 
$r\in [0,\epsilon_0/2]$, and $\alpha_r\leq 0$.
We ask that $\beta_r\equiv 0$ throughout zone 1.
\item
Zone 2: For $r\in [\epsilon_0,R/2]$ we ask that $\alpha\equiv \hat\alpha$.
\item
Zone 3: For $r\in [R/2,R]$, we ask that $0\leq \alpha_r \leq r$
and $\beta(s,r)=\log\cos r$ as $\beta$ would be on the standard round spherical metric
\eqref{sph_metric}.
\item
For $r\in [R,\frac{\pi}{2}]$ we simply ask that $\alpha\equiv 1$ and 
$\beta(s,r)=\log\cos r$ to recover the round metric.
\end{enumerate}
This is all we ask for $\alpha(r)$, but we can only choose $\alpha$ in zone 3 at this stage because we have not yet specified $\epsilon_0$.
Note that the zone 3 constraint that $\alpha_r \leq r$ is easy to achieve by virtue of our insistence that $\hat\alpha\geq 1-\frac{R^2}{5}$ in \eqref{hat_alpha_choice}.
Additional constraints will be imposed on $\beta$ in due course.

\begin{figure}
\centering
\begin{tikzpicture}

\draw [->] (0,0) -- (11,0) node[below]{$r$} ;
\draw (10,0.1) -- (10,-0.1) node[below]{$\frac{\pi}{2}$};
\draw [->] (0,0) -- (0,4) node[left]{$\alpha(r)$} ;

\draw [thick] (4,3) -- (10,3);
\draw (4,0.1) -- (4,-0.1) node[below]{$R$};

\draw [thick] (0,3) node[left]{$1$} -- (0.3,3);
\draw [thick] (0.3,3) .. controls (0.5,3) and (0.4,2.5) .. (0.6,2.5);
\draw (0.6,0.1) -- (0.6,-0.1) node[below]{$\epsilon_0$};
\draw (0.3,0.1) -- (0.3,-0.1); 

\draw [thick] (0.6,2.5) -- (2,2.5);
\draw [dotted] (0.6,2.5) -- (0,2.5) node[left]{$\hat\alpha$};
\draw [dotted] (0,9/4) -- (11,9/4) node[right]{$\frac{3}{4}$};
\draw (2,0.1) -- (2,-0.1) node[below]{$\frac{R}{2}$};

\draw [thick] (2,2.5) .. controls (3,2.5) and (3,3) .. (4,3);

\draw [<->] (0,-1) node[below left]{Zone: } -- node[below]{$1$} (0.6,-1) ;
\draw [<->] (0.6,-1) -- node[below]{$2$} (2,-1) ;
\draw [<->] (2,-1) -- node[below]{$3$} (4,-1) ;

\end{tikzpicture}
\caption{Graph of the warping function $\alpha(r)$}
\label{fig:alpha}
\end{figure}
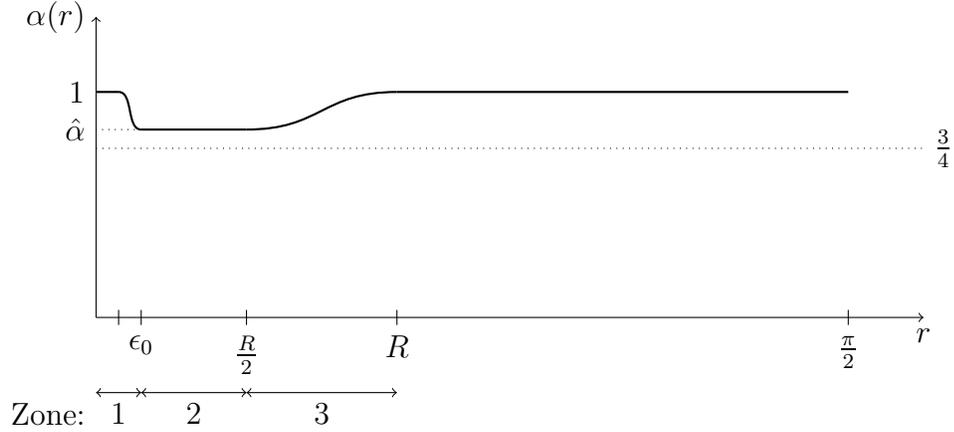

A short computation 
reveals that the sectional curvatures of $g$ are given by 
$$K_{ri}=\alpha-{\textstyle\frac12}\alpha_r\cot r$$
$$K_{rs}=-{\textstyle\frac12}\alpha_r \beta_r-\alpha \beta_{rr}-\alpha (\beta_r)^2$$
$$K_{is}=-\alpha \beta_r \cot r $$
$$K_{ij}=\csc^2 r - \alpha\cot^2 r,$$
%
and these can be combined to give the scalar curvature
\begin{equation}
\label{scalar_eq}
\mathcal{R}=2K_{rs}+2(n-2)K_{ri}+2(n-2)K_{is}+(n-2)(n-3)K_{ij}.
\end{equation}
This formula can be computed directly using standard formulae for  warped products, or weighted scalar curvature, but we find it illuminating to work directly with the sectional curvatures.

We will establish the  bounds  in Table \ref{table:sc} in the three zones.

\begin{table}[ht]
\caption{Sectional curvature lower bounds}
\centering
\begin{tabular}{c || c c c}
\hline\hline
Section & Zone 1 & Zone 2 & Zone 3 \\ 
\hline 
$K_{ri}$ & $\frac14$ & $\frac14$ & $\frac14$  \\ [1ex]
$K_{rs}$ & $0$ & $-\alpha (\beta_{rr}+(\beta_r)^2)$ & 0 \\ [1ex]
$K_{is}$ & 0 & 
$-\alpha \beta_r \cot r$
& 0 \\ [1ex]
$K_{ij}$ & 0 & $\frac{1-{\hat\alpha} }{r^2}$ & 0 \\ [1ex]
\hline
\end{tabular}
\label{table:sc}
\end{table}

By our assumptions, it is easy to verify that $K_{ri}\geq \frac14$ everywhere. 
Indeed, outside zone 3 we have $-{\textstyle\frac12}\alpha_r\cot r \geq 0$
so $K_{ri}\geq \alpha\geq \frac34$,
whereas in zone 3, because $\alpha_r\cot r\leq \frac{\alpha_r}{r}\leq 1$
we have $K_{ri}\geq \alpha-\frac12\geq \frac14$.
Therefore $K_{ri}$ always serves to make the scalar curvature more positive.

Second, our assumptions imply that $K_{rs}\geq 0$ outside zone 2.
Indeed, in zone 1 we have $\beta_r\equiv 0$, so $K_{rs}\equiv 0$ there.
In zone 3, because $\beta=\log\cos r$, the expression for the sectional curvature
simplifies to $K_{rs}=\alpha+\frac12 \alpha_r\tan r$, and both terms are nonnegative.
In zone 2, because $\alpha_r\equiv 0$, the expression for the sectional curvature simplifies to
\begin{equation}
\label{Krs}
K_{rs}=-\alpha (\beta_{rr}+(\beta_r)^2)
\end{equation}
but we will need to carefully absorb this (possibly large and negative)  term elsewhere.

Third, our assumptions imply that $K_{is}\geq 0$ outside zone 2.
Indeed, in zone 1 we have $\beta_r\equiv 0$, so $K_{is}\equiv 0$ there.
In zone 3, because $\beta=\log\cos r$, the expression for the sectional curvature
simplifies to $K_{is}\equiv\alpha\geq 0$.

Finally, we observe that 
\begin{equation}
\label{Kij}
K_{ij}=\frac1{\sin^2 r}\left( 1-\alpha \cos^2 r\right)\geq 0.
\end{equation}

At this point we have shown that outside zone 2, each of the sectional curvatures
considered is nonnegative and $K_{ri}\geq \frac14$, and so the scalar curvature
\eqref{scalar_eq}
of $g$ is positive outside zone 2.

Within zone 2 we must be a little more careful. 
First we observe that within zone 2 we can develop estimate \eqref{Kij} to
$$K_{ij}
\geq \frac{1-{\hat\alpha} }{r^2} ,$$
which is very large for small $r$, and we will arrange that it dominates the curvatures
$K_{rs}$ and $K_{is}$.
 
To this end, it will be useful to construct a cut-off function adapted to this situation. 
We start by picking a non-decreasing function $\eta\in C^\infty(\R,[0,1])$ with 
$\eta(x)=0$ for $x\leq \frac12$ and $\eta(x)=1$ for $x\geq 1$, as in Figure \ref{fig:eta}. We may as well assume that 
$\eta'\leq C$ and $|\eta''|\leq C$, for some universal constant $C$, so that our estimates do not depend on the choice of $\eta$ we make.

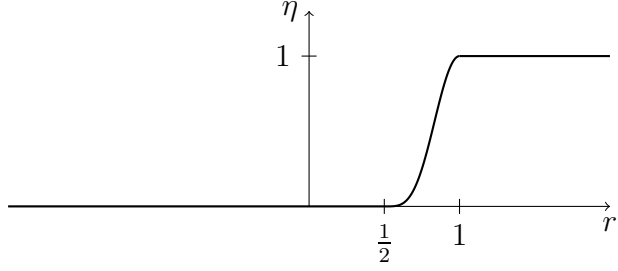
\begin{figure}
\centering
\begin{tikzpicture}[xscale=2,yscale=2]

\draw [->] (-2,0) -- (2,0) node[below]{$r$} ;
\draw (1,0.05) -- (1,-0.05) node[below]{$1$};
\draw (0.5,0.05) -- (0.5,-0.05) node[below]{$\frac{1}{2}$};
\draw [->] (0,0) -- (0,1.3) node[left]{$\eta$} ;
\draw (0.05,1) -- (-0.05,1) node[left]{$1$};

\draw [thick] (-2,0) -- (0.5,0);
\draw [thick] (1,1) -- (2,1);

\draw[thick, domain=0.5:1, samples=500] 
plot (\x, {
((1-cos((\x-0.5)*360))/2)^2
});

\end{tikzpicture}
\caption{Graph of the smooth cut-off function $\eta$}
\label{fig:eta}
\end{figure}

For a tiny $\delta\in (0,1)$ to be chosen in a moment, define the scaled cut-off
$\varphi\in C^\infty(\R,[0,1])$ by
$$\varphi(r):=\eta\left(1+\delta\log\frac{2r}{R}\right),$$
as illustrated in Figure \ref{fig:vph}.

\begin{figure}
\centering
\begin{tikzpicture}[xscale=2,yscale=2]

\draw [->] (-2,0) -- (2,0) node[below]{$r$} ;
\draw (1,0.05) -- (1,-0.05) node[below]{$\frac{R}{2}$};
\draw [->] (0,0) -- (0,1.3) node[left]{$\varphi$} ;
\draw (0.05,1) -- (-0.05,1) node[left]{$1$};

\draw ({exp((0.5-1)/0.15)},0.05) -- ({exp((0.5-1)/0.15)},-0.05) node[below]{$\epsilon_0$};

\draw [thick] (-2,0) -- ({exp((0.5-1)/0.15)},0);
\draw [thick] (0.98,1) -- (2,1);

\draw[thick, domain=0.5:1, samples=200] 
plot ({exp((\x-1)/0.15)}, {
((1-cos((\x-0.5)*360))/2)^2
});

\end{tikzpicture}
\caption{Graph of the smooth cut-off function $\varphi$}
\label{fig:vph}
\end{figure}
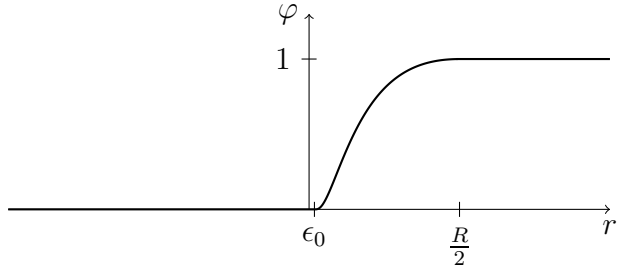

Then $\varphi$ is also non-decreasing, with 
$\varphi(r)=0$ for 
\begin{equation}
\label{ep0_def}
\textstyle
r\leq \epsilon_0:=e^{-\frac{1}{2\delta}}\frac{R}{2},
\end{equation}
i.e. in zone 1,
but now with $\varphi(r)=1$ for $r\geq \frac{R}{2}$, i.e. in zone 3.
We observe that 
$$0\leq \varphi'(r) \leq \frac{C\delta}{r}
\qquad\text{ and }\qquad
|\varphi''(r)|\leq \frac{C\delta}{r^2}$$
for some universal $C$.
Define
\begin{equation}
\label{beta_def}
\beta(s,r)=\varphi(r)\log\cos r + (1-\varphi(r))f(s,0),
\end{equation}
where $f(s,0)$ refers to the conformal scaling factor from the theorem, restricted to the great circle $\{r=0\}$.
That is, we interpolate between $f(s,0)$ near $r=0$ and $\log\cos r$ for $r\geq \frac{R}{2}$, and note that for  $0<r\leq R<\frac{1}{100}$, we have 
$$\log\cos r \geq \log\cos \frac{1}{100} > -1 \geq f(s,0)$$ 
and so 
\begin{equation}
\label{beta_bd}
f(s,0)\leq \beta(s,r)\leq \log\cos r.
\end{equation}
Then in zone 1 we have $\beta(s,r)=f(s,0)$ (so $\beta_r=0$ as specified) and 
in zone 3 we have $\beta(s,r)=\log\cos r$, also as specified.
We compute 
\begin{equation}
\label{beta_r_form}
\beta_r=\varphi'(r)\big[\log\cos r - f(s,0)\big]   -\varphi(r)\tan r
\end{equation}
and so in zone 2 we have 
$$\beta_r \leq -\varphi'(r) f(s,0)  \leq \frac{C\delta \bar{f}}{r}$$
for universal $C$ 
(recall $-f\leq \bar{f}$). 
This gives us a lower bound in zone 2 of $K_{is}\geq -\frac{C\delta \bar{f}}{r^2}$.
In the other direction, we have
$$\beta_r \geq \varphi'(r)\big[\log\cos R +1\big] -\tan R \geq -\tan R 
\geq -\frac{1}{10},$$
say, because $f\leq -1$ and $R<\frac{1}{100}$.
Meanwhile
\begin{equation}
\begin{aligned}
\beta_{rr} &=
\varphi''(r)\big[\log\cos r - f(s,0)\big]   
-2\varphi'(r)\tan r   
-\varphi(r)\sec^2 r\\
&\leq |\varphi''(r)|\big[ -\log\cos R + \bar{f}\big] 
\end{aligned}
\end{equation}
and so in zone 2 we have
$$\beta_{rr}\leq \frac{C\delta \bar{f}}{r^2},$$
because $-\log\cos R\leq -\log\cos \frac{1}{100} < 1 \leq \bar{f}$.
Combining, we find from \eqref{Krs} that 
$$K_{rs}=-\alpha (\beta_{rr}+(\beta_r)^2)\geq - \frac{C\delta \bar{f}^2}{r^2}-\frac{1}{100}$$
in zone 2, where we are using that $\delta<1$ and $\bar{f}\geq 1$, and $C$ is a different universal constant each time.
This new information allows us to revise our sectional curvature lower bounds 
as in Table \ref{table:sc2}.
\begin{table}[ht]
\caption{Sectional curvature lower bounds - revised}
\centering
\begin{tabular}{c || c c c}
\hline\hline
Section & Zone 1 & Zone 2 & Zone 3 \\ 
\hline 
$K_{ri}$ & $\frac14$ & $\frac14$ & $\frac14$  \\ [1ex]
$K_{rs}$ & $0$ & $-\frac{C\delta \bar{f}^2}{r^2}-\frac{1}{100}$ & 0 \\ [1ex]
$K_{is}$ & 0 & $-\frac{C\delta \bar{f}}{r^2}$ & 0 \\ [1ex]
$K_{ij}$ & 0 & $\frac{1-{\hat\alpha} }{r^2}$ & 0 \\ [1ex]
\hline
\end{tabular}
\label{table:sc2}
\end{table}
Keeping in mind that the scalar curvature is given by \eqref{scalar_eq},
we find that for sufficiently small $\delta$, depending only on $\bar{f}$ and ${\hat\alpha}$,
we have $\mathcal{R}>0$. Fixing such a $\delta>0$, we obtain Part (1) of the lemma.

Note that only now that we have picked $\delta$ is $\epsilon_0$ determined by 
\eqref{ep0_def} and $\beta$ is fixed. That allows us to fix a specific function 
$\alpha$ in zone 1, i.e. any non-increasing function that is identically $1$ on $[0,\epsilon_0/2]$ and ends up at $\hat\alpha$ by $r=\epsilon_0$.

Part (2) of the lemma is immediate by construction. 
To see part (3) of the lemma, note that the ansatz \eqref{ansatz} and 
\eqref{beta_bd} means that the only part of the metric that can expand compared with the standard metric $g_0$ is the $dr^2$ part. Indeed, because $\alpha(r)\geq \hat\alpha$, 
we have 
$g\leq \frac{1}{\hat\alpha}g_{0}$, and because
$\hat\alpha\geq 1-\frac{\e}{2}$ by \eqref{hat_alpha_choice} 
we deduce that 
$g\leq (1+\e)g_{0}$ throughout ${\mathbb{S}^{n}}$, as claimed in Part (3).
(Note that  $\frac{1}{1-\frac{\e}{2}}\leq 1+\e$ when $\e\in (0,1)$.)

To see Part (4), first note that by Part (2) we need only consider the inequality within $\mathcal{C}_R$. 
Within $\mathcal{C}_R$, the metric $g$ is larger even than the round metric $g_0$, let alone $e^{2f}g_0$, except possibly in the $s$ direction. Indeed,
by the ansatz \eqref{ansatz}, the second inequality \eqref{beta_bd} for $\beta$, and the fact that $\alpha\leq 1$, we see that
$$g\geq e^{2(\beta(s,r)-\log\cos r)}g_0,$$
and so because $-\log\cos r\geq 0$, and by the first inequality \eqref{beta_bd} for $\beta$, we obtain
$$g\geq e^{2f(s,0)}g_0$$
within $\mathcal{C}_R$.
The uniform continuity of \eqref{R_unif_cty} tells us that 
$$\textstyle |f(s,0)-f(s,r)|<\frac12\log(1+\e),$$ 
and in particular that
$$e^{2f(s,0)}\geq e^{2f(s,r)-\log(1+\e)}$$
and hence
$$(1+\e)g\geq e^{2f}g_0$$
as claimed in Part (4).

Part (5) is immediate from the ansatz \eqref{ansatz} and the definition \eqref{beta_def}
of $\beta$.
\end{proof}


\section{Construction of the approximating metrics}

Given a metric $h=e^{2f}g_{0}$ and two nearby points $x,y\in {\mathbb{S}^{n}}$, the (shorter) arc of the great circle connecting $x$ and $y$ has $h$-length roughly equal to $d_h(x,y)$ as articulated in the following lemma.
\begin{lma}
\label{almost_geod_lem}
Suppose $h=e^{2f}g_{0}$, where $f\in C^0(\mathbb{S}^{n})$,
and suppose $\e>0$. Then there exists $\delta\in (0,1)$ such that if $x,y\in \mathbb{S}^{n}$ with
$d_h(x,y)\leq 2\delta$, then
$$L_h(\mathcal{C}_{x,y})\leq (1+\e)d_h(x,y).$$
where $\mathcal{C}_{x,y}$ is a minimising geodesic, with respect to the round metric, connecting $x$ and $y$.
\end{lma}


\begin{proof}
Let $\bar{f}:=\max (-f)$. Because $f$ is continuous on a compact space, it is uniformly continuous, and so there exists $\delta>0$ so that 
\begin{equation}
\label{unif_cty}
\textstyle
d_{g_{0}}(\tilde x,\tilde y)\leq 2\delta e^{\bar{f}}
\qquad\text{ implies}\qquad
|f(\tilde x)-f(\tilde y)|<\frac12\log(1+\e).
\end{equation}

Notice that $d_{g_{0}}(x,y)\leq e^{\bar{f}}d_h(x,y)\leq 2\delta e^{\bar{f}}$, and so we can apply \eqref{unif_cty} with $\tilde x=x$ and $\tilde y$ an arbitrary point on $\mathcal{C}_{x,y}$ to obtain
\begin{equation}
\begin{aligned}
L_h(\mathcal{C}_{x,y}) &= L_{e^{2f}g_{0}}(\mathcal{C}_{x,y})\\
&\leq (1+\e)^\frac12 L_{e^{2f(x)}g_{0}}(\mathcal{C}_{x,y})\\
&= (1+\e)^\frac12 d_{e^{2f(x)}g_{0}}(x,y).
\end{aligned}
\end{equation}
Similarly, we can apply \eqref{unif_cty} with $\tilde x=x$ and $\tilde y$ an arbitrary point on a minimising $h$-geodesic $\ga_{x,y}$ from $x$ to $y$ to continue
\begin{equation}
\begin{aligned}
L_h(\mathcal{C}_{x,y}) 
&\leq (1+\e)^\frac12 L_{e^{2f(x)}g_{0}}(\gamma_{x,y})\\
&\leq (1+\e) L_{h}(\gamma_{x,y})\\
&=(1+\e)d_h(x,y)
\end{aligned}
\end{equation}
\end{proof}

\begin{proof}[Proof of Theorem~\ref{main}]
By mollification, it suffices to prove the theorem for smooth $f$.
By scaling, it suffices to consider the case 
$f\leq -1$. To see this, suppose  $f\leq L$ for some $L\in \mathbb{R}$. We may apply the assertion to $\tilde h:=e^{2(f-L-1)}g_{0}$ to obtain a sequence $\tilde g_i$ approaching $\tilde h$  with $\mathcal{R}(\tilde g_i)>0$. The rescaled sequence $g_i=e^{2(L+1)}\tilde g_i$  retains the positive scalar curvature and approaches the original desired destination metric $e^{2f}g_{0}$.
As earlier, we adopt the shorthand 
$$h:=e^{2f}g_{0}.$$

Suppose $\e\in (0,1)$ is arbitrary. The theorem will be proved if we can construct a metric $g_\e$ on ${\mathbb{S}^{n}}$ with the properties that
\begin{enumerate}
\item
$\mathcal{R}(g_\e)>0$
\item
$h\leq (1+\e)g_\e$ 
\item
$d_{g_\e}(x,y)\leq d_h(x,y)+C\e$ for some universal $C<\infty$ and all $x,y\in {\mathbb{S}^{n}}$.
\end{enumerate}

For the given $\e$, let $\delta>0$ be as in Lemma \ref{almost_geod_lem}.
We abbreviate $\eta=\e\delta/6$.

Our first task is to find a large number of disjoint great circles so that every pair of points in the sphere is close to one of them.
To do this, let $P\in\mathbb{N}$ be the largest number of pairwise disjoint $g_{0}$-geodesic balls of radius $\eta$ that can be squeezed into the sphere $\mathbb{S}^{n}$, 
and make a choice $\{B_{\eta}(w_i)\}_{i=1}^{i=P}$ of such balls. Thus every point in the sphere lies within a distance $2\eta$ of one such point $w_i$.
Taking each unordered pair $\{w_i,w_j\}$ of distinct points in turn, we choose a great circle that passes within a distance $\eta$ of both $w_i$ and $w_j$, but does not intersect any of the previously picked great circles. We end up with a finite set $\{\mathcal{C}_\gamma\}_{\gamma=1}^{N}$ of great circles so that given any pair of distinct points $x$ and $y$ in the sphere, both $x$ and $y$ lie within a $g_{0}$-distance $3\eta$ of a common great circle $\mathcal{C}_\gamma$. 

Pick $R\in (0,\frac{1}{100})$ sufficiently small so that the $R$-tubular neighbourhoods of the great circles $\mathcal{C}_\gamma$ are pairwise disjoint. 
Each of the tubular neighbourhoods can now be excised and replaced with the corresponding metric constructed in Lemma \ref{building_block_lem}. We call the resulting metric
$g_\e$ and claim that it enjoys the properties (1)-(3) listed above.

The positive scalar curvature follows immediately from Lemma \ref{building_block_lem},
as does the inequality $h\leq (1+\e)g_\e$. 
It remains to prove that for arbitrary $x,y\in {\mathbb{S}^{n}}$ we have
$d_{g_\e}(x,y)\leq d_h(x,y)+C\e$.

We first claim that if $d_h(\tilde x,\tilde y)\leq\delta$ then 
\begin{equation}
\label{claim}
d_{g_\e}(\tilde x,\tilde y)\leq d_h(\tilde x,\tilde y)+C\delta\e
\end{equation}
for universal $C$.
To prove the claim, we can adjust $\tilde x$ and $\tilde y$ to nearby points 
$a$ and $b$ (respectively) on one of the great circles
$\mathcal{C}$, with $d_{g_{0}}(\tilde x,a)\leq 3\eta$ and $d_{g_{0}}(\tilde y,b)\leq 3\eta$, so
\begin{equation}
\begin{aligned}
d_h(a,b) &\leq d_h(a,\tilde x)+d_h(\tilde x,\tilde y)+d_h(\tilde y,b)\\
&\leq 3\eta + \delta +3\eta\\
&\leq 2\delta
\end{aligned}
\end{equation}
where we recall that $d_h<d_{g_{0}}$ and 
$\eta=\e\delta/6\leq \delta/6$.
By 
the fact that 
$g_\e\leq (1+\e)g_{0}$ (which also follows from Lemma \ref{building_block_lem}) and so $d_{g_\e}\leq 2d_{g_{0}}$,
we compute
\begin{equation}
\begin{aligned}
d_{g_\e}(\tilde x,\tilde y) 
&\leq d_{g_\e}(\tilde x,a)+d_{g_\e}(a,b)+d_{g_\e}(b,\tilde y)\\
&\leq 2d_{g_{0}}(\tilde x,a)+L_{g_\e}(\mathcal{C}_{a,b})+2d_{g_{0}}(b,\tilde y)\\
&\leq 6\eta + L_{h}(\mathcal{C}_{a,b})+6\eta\qquad\text{(because $L_{h}(\mathcal{C}_{a,b})=L_{g_\e}(\mathcal{C}_{a,b})$)}\\
&\leq 12\eta + (1+\e)d_h(a,b)\qquad\text{(by Lemma \ref{almost_geod_lem})}\\
&\leq 12\eta + (1+\e)[d_h(a,\tilde x)+d_h(\tilde x,\tilde y)+d_h(\tilde y,b)]\\
&\leq C\eta + (1+\e)d_h(\tilde x,\tilde y)\\
&\leq d_h(\tilde x,\tilde y) + C\e\delta\\
\end{aligned}
\end{equation}
where $C$ is universal and we have used that $\eta=\e\delta/6$ and $d_h(\tilde x,\tilde y)\leq \delta$. This is precisely the claim \eqref{claim}.

Now let $\gamma: [0,d_h(x,y)]\to {\mathbb{S}^{n}}$ be a unit-speed minimising geodesic, with respect to $h$, connecting $x$ and $y$. 
We divide $\gamma$ up into $\delta$-chunks by defining
$x_k=\gamma(k\delta)$ for $k=0,...,\ell$, where $\ell=[d_h(x,y)/\delta]\leq C/\delta$, and add a final point
$x_{\ell+1}:=y$.
Then
\begin{equation}
\label{tri_ineq}
\begin{aligned}
d_{g_\e}(x,y) &\leq \sum_{k=0}^{\ell}d_{g_\e}(x_k,x_{k+1})\\
&\leq C\ell\delta\e+\sum_{k=0}^{\ell}d_{h}(x_k,x_{k+1})\qquad\text{(by the claim)}\\
&\leq C\e + d_{h}(x,y)
\end{aligned}
\end{equation}
where $C$ is always universal. Thus all properties (1)-(3) claimed above have been established.
\end{proof}

\section{The torus case}
\label{torus_sect}

In this section we outline the changes to the argument required in order to adapt the proof of Theorem \ref{main} to prove Theorem \ref{main_torus}.

In this case, instead of criss-crossing the sphere with great circles, we consider closed geodesics in the torus that are so numerous that any two points in the torus can be perturbed to lie on a common such geodesic.
We can then excise disjoint tubular neighbourhoods of these geodesics and replace them by shrunk metrics somewhat as in the spherical case of Theorem \ref{main}. One difference is that unlike in the case of the sphere, the torus has no positive scalar curvature with which errors can be absorbed. Therefore we must allow a small drop in the scalar curvature. The analogue of Lemma \ref{building_block_lem} is then:
\begin{lma}
\label{building_block_lem_torus}
Suppose $n\geq 4$, $\mathcal{C}$ is a simple closed geodesic in a flat $n$-dimensional manifold $M$, and $R>0$ is sufficiently small so that $\mathcal{C}$ admits a 
$R$-tubular neighbourhood $\mathcal{C}_R$.
Suppose $f\in C^\infty(M)$ with $f\leq -1$, and 
define $\bar{f}=\max (-f)$, so that $f\in [-\bar{f},-1]$.
Then for each $\e\in (0,1)$ we can find a new smooth Riemannian metric $g$ on 
$M$ with the properties that 
\begin{enumerate}
\item
$\mathcal{R}(g)\geq -\e$
\item
$g=g_{0}$ outside $\mathcal{C}_R$
\item
$g\leq (1+\e)g_{0}$ throughout $M$
\item
$e^{2f}g_{0}\leq (1+\e)g$, and in particular 
$e^{-2\bar{f}}g_{0}\leq (1+\e)g$
\item
The metrics on $\mathcal{C}$ induced by restricting 
$e^{2f}g_{0}$ and $g$ are equal. That is, they agree on the length of the vector 
$\partial_s$.
\end{enumerate}
\end{lma}

In order to adapt the proof of Lemma \ref{building_block_lem}, we replace the ansatz
\eqref{ansatz} by
\begin{equation}
\label{ansatz_torus}
g=\frac{dr^2}{\alpha(r)}+ r^2 \cdot h_{\mathbb{S}^{n-2}}+e^{2\beta (s,r)} \cdot  ds^2.
\end{equation}
The flat torus case would be $\alpha\equiv 1$ and $\beta\equiv 0$.
We can now compute the sectional curvatures to be
$$K_{ri}=-\frac{\alpha_r}{2r}$$
$$K_{rs}=-{\textstyle\frac12}\alpha_r \beta_r-\alpha \beta_{rr}-\alpha (\beta_r)^2$$
$$K_{is}=-\frac{\alpha \beta_r}{r}$$
$$K_{ij}=\frac{1 - \alpha}{r^2}.$$
We can choose $\alpha$ essentially as before, although we drop the condition 
that $\hat\alpha\geq 3/4$ and ask that it is much closer to $1$ by insisting 
that $R\leq 1$ and defining
\begin{equation}
\label{hat_alpha_choice_torus}
\textstyle\hat\alpha=1-\frac{\e R^2}{5n} <1,
\end{equation}
which ensures, in particular, that $\hat\alpha\geq 1-\frac{\e}{2}$, but also allows
us to insist that 
$$\alpha_r\leq \frac{\e r}{n} $$ 
in zone 3, which forces a lower bound $K_{ri}\geq -\frac{\e}{2n}$
there.
%
%

We retain the condition $\beta_r\equiv 0$ in zone 1, and 
in zone 3 and beyond we now simply require $\beta\equiv 0$.
Thus $K_{rs}\equiv 0$ in zones 1 and 3 as before.
The general formula \eqref{beta_def} for $\beta$ simplifies to
\begin{equation}
\label{beta_def_torus}
\beta(s,r)=(1-\varphi(r))f(s,0),
\end{equation}
so the formula \eqref{beta_r_form} for $\beta_r$ simplifies to
$$\beta_r= -\varphi'(r))f(s,0).$$
The resulting upper bound for $\beta_r$ remains identical but the lower bound improves to $\beta_r\geq 0$.
The bound for $\beta_{rr}$ remains the same, but with a simpler derivation.

The resulting sectional curvature lower bounds are summarised in 
Table \ref{table:sc_torus}.
\begin{table}[ht]
\caption{Torus sectional curvature lower bounds}
\centering
\begin{tabular}{c || c c c}
\hline\hline
Section & Zone 1 & Zone 2 & Zone 3 \\ 
\hline 
$K_{ri}$ & $0$ & $0$ & $-\frac{\e}{2n}$  \\ [1ex]
$K_{rs}$ & $0$ & $-\frac{C\delta \bar{f}^2}{r^2}$ & 0 \\ [1ex]
$K_{is}$ & 0 & $-\frac{C\delta \bar{f}}{r^2}$ & 0 \\ [1ex]
$K_{ij}$ & 0 & $\frac{1-{\hat\alpha} }{r^2}$ & 0 \\ [1ex]
\hline
\end{tabular}
\label{table:sc_torus}
\end{table}
Given these bounds and the formula for the scalar curvature \eqref{scalar_eq}, 
for the given $\hat \alpha<1$, we can then pick $\delta>0$ sufficiently small so that the scalar curvature is positive in zone 2. Then \eqref{scalar_eq} implies that 
$\mathcal{R}\geq 0$ outside zone 3, while in zone 3 we have
$$\mathcal{R}\geq 2(n-2)K_{ri}\geq 
2(n-2)(-\frac{\e}{2n})
\geq -\e.$$


\begin{thebibliography}{1000}

\bibitem{bamler} R. Bamler,
\emph{A Ricci flow proof of a result by Gromov on lower bounds for scalar curvature.}
Math. Res. Letters, {\bf 23} (2016) 325--337.


\bibitem{BasilioDodziukSormani}J. Basilio, J. Dodziuk and  C. Sormani, \emph{Sewing Riemannian manifolds with positive scalar curvature}. J. Geom. Anal. {\bf 28} (2018), no. 4, 3553--3602.

\bibitem{BasilioKazarasSormani}J. Basilio, D. Kazaras and C. Sormani, \emph{An intrinsic flat limit of Riemannian manifolds with no geodesics}. Geom. Dedicata {\bf 204} (2020), 265--284.


\bibitem{BasilioSormani}J. Basilio and C. Sormani, \emph{Sequences
 of three dimensional manifolds with positive scalar curvature} Differential
 Geometry and its Applications, {\bf 77} (2021).

\bibitem{BBI} D. Burago, Y. Burago and S. Ivanov,
`A course in metric geometry.' Graduate studies in math. {\bf 33} AMS 2001.


\bibitem{ChuLee} J. Chu and M.-C. Lee, \emph{Conformal tori with almost non-negative scalar curvature}. Calc. Var. Partial Differential Equations {\bf 61}:114 (2022).

\bibitem{Burkhardt}P. Burkhardt-Guim, \emph{Pointwise lower scalar curvature bounds for $C^0$ metrics via regularizing
Ricci flow.} Geom. Funct. Anal. {\bf 29} (2019), no. 6, 1703--1772.






\bibitem{gromov} M. Gromov, 
\emph{Dirac and Plateau billiards in domains with corners.}
Cent. Eur. J. Math. 
{\bf 12} (2014) 1109--1156.

\bibitem{gromov_survey} M. Gromov, 
\emph{Four Lectures on Scalar Curvature.} 
`Perspectives in Scalar Curvature.' World Scientific (2023) 1--514.


\bibitem{HuangLeeSormani}L.-H. Huang, D. A. Lee and C. Sormani,
\emph{Intrinsic flat stability of the positive mass theorem for graphical hypersurfaces of Euclidean space}. J. Reine Angew. Math., 727:269--299, 2017.


\bibitem{KazarasXu} D. Kazaras and K. Xu,
\emph{Drawstrings and flexibility in the Geroch conjecture} Version 2, (2023).
\href{https://arxiv.org/abs/2309.03756}{arXiv:2309.03756}





\bibitem{LNN} M.-C. Lee, A. Naber and R. Neumayer,
\emph{$d_p$ convergence and $\e$-regularity theorems for entropy and scalar curvature lower bounds.} Geom. Topol.
\textbf{27:1} (2023) 227--350.

\bibitem{LT1} M.-C. Lee and P. M. Topping, 
\emph{Time zero regularity of Ricci flow.} 
International Mathematics Research Notices, 
\textbf{2023} (2023) 21167--21179.


\bibitem{Sormani} C. Sormani, \emph{Conjectures on Convergence and Scalar Curvature}, 
`Perspectives in Scalar Curvature.' World Scientific (2023) 645--722.


\bibitem{villani} C. Villani.
`Optimal transport, old and new.' Springer (2009).

\end{thebibliography}
\end{document}